\documentclass[10pt,draftcls,onecolumn]{IEEEtran}
\usepackage{epsfig,amssymb,latexsym,color,pifont,colordvi,multicol}
\usepackage[cmex10]{amsmath}
\usepackage{moreverb}
\usepackage{subfigure,mathptmx,times,verbatim}
\usepackage{subfigure}

\definecolor{backgrey}{rgb}{0.86,0.86,0.86}
\definecolor{dblue}{rgb}{0,0.0,0.5}
\definecolor{dred}{rgb}{0.4,0.2,0}
\definecolor{dgreen}{rgb}{0.0,0.5,0}

\newcommand{\captionfonts}{\small}
\makeatletter  
\long\def\@makecaption#1#2{%
  \vskip\abovecaptionskip
  \sbox\@tempboxa{{\captionfonts #1: #2}}%
  \ifdim \wd\@tempboxa >\hsize
    {\captionfonts #1: #2\par}
  \else
    \hbox to\hsize{\hfil\box\@tempboxa\hfil}%
  \fi
  \vskip\belowcaptionskip}
\makeatother   
\newtheorem{theorem}{Theorem}

\newtheorem{assumption}[theorem]{Assumption}

\newtheorem{remark}[theorem]{Remark}

\newtheorem{definition}[theorem]{Definition}

\newtheorem{lemma}[theorem]{Lemma}

\newenvironment{proof}[1][Proof]{\textbf{#1.} }{\ \hspace*{\fill} \rule{0.5em}{0.5em}}

\title{Control of systems in Lure form over erasure channels}
\author{\quad Amit Diwadkar \quad Sambarta Dasgupta \quad Umesh Vaidya 
\thanks{A. Diwadkar is a post-doctoral researcher with the Department of Electrical and
Computer Engineering, Iowa State University, Ames, IA, 50011
diwadkar@iastate.edu}
\thanks{S. Dasgupta is a graduate student with the Department of Electrical and
Computer Engineering, Iowa State University, Ames, IA, 50011
dasgupta@iastate.edu}
\thanks{U. Vaidya is with the Department of Electrical and
Computer Engineering,
Iowa State University,
Ames, IA, 50011 ugvaidya@iastate.edu}%
}
\begin{document}
\maketitle \thispagestyle{empty} \pagestyle{empty}
\begin{abstract}                          
In this paper, we study the problem of control of discrete-time nonlinear systems in Lure form over erasure channels at the input and output. The input and output channel uncertainties are modeled as Bernoulli random variables. The main results of this paper provide sufficient condition for the mean square exponential stability of the closed loop system expressed in terms of statistics of channel uncertainty and plant characteristics. We also provide synthesis method for the design of observer-based controller that is robust to channel uncertainty. To prove the main results of this paper, we discover a stochastic variant of the well known Positive Real Lemma and principle of separation for stochastic nonlinear system. Application of the results for the stabilization of system in Lure form over packet-drop network is discussed. Finally a result for state feedback control of a Lure system with a general multiplicative uncertainty at actuation is discussed.
\end{abstract}


\section{Introduction}
The problem of control and estimation over stochastic input and output channels is of importance in control of systems over networks \cite{networksystems_specialissue}. Specifically systems with packet drop communication channels at input and output, modeled as Bernoulli erasure channels has garnered much attention \cite{networksystems_martin_gupta,network_basar_tamar,Luca_03_kalmanfiltering,scl04,murray_epstein,networksystems_foundation_sastry}. These communication channels are modeled as a multiplicative uncertainty with Bernoulli random variable. Majority of the research regarding control and estimation over erasure communication channels has been done for linear time invariant (LTI) systems. Control and estimation over fading channels was studied by modeling the problem in the robust control framework \cite{scl04}. An LMI formulation was provided for packet-erasure channels and a necessary and sufficient condition was provided for single input single output (SISO) systems. The sufficient conditions relates a tradeoff between packet-erasure probability and product of unstable eigenvalues of the system matrix, which give the volume expansion rate of the system. linear quadratic regulator control is studied for LTI systems with erasure communication channels for both Transmission Control Protocol (TCP) and User Datagram Protocol (UDP) \cite{network_basar_tamar}. Kalman filtering over lossy communication channels modeled as packet-drop channels for LTI systems has also been studied for bounded variance stability conditioned upon the channel probabilities \cite{Luca_03_kalmanfiltering,networksystems_foundation_sastry}. A necessary and sufficient condition is derived for the channel communication probability based on the maximum rate of expansion of the linear systems as given by the maximum eigenvalue of the system. The stability condition has also been relaxed to stability in probability for packet-erasure communication channels \cite{murray_epstein}. These results were first extended for nonlinear systems for the problems of observation and state feedback control for linear time varying (LTV) systems and ultimately for nonlinear systems \cite{vaidya_erasure_scl,amit_observation_jouranl,amit_ltv_jouranl}. The framework based on the theory of random dynamical system was adapted in \cite{Vaidya_erasure_stablization,vaidya_erasure_scl} to prove necessary and sufficient conditions for scalar systems. The design problem of a mean square stable state feedback control for nonlinear systems \cite{vaidya_erasure_scl} provides a necessary and sufficient condition that mimics the conditions for mean square stable state feedback control of LTI systems \cite{scl04}. Similar random dynamical system framework was utilized in \cite{amit_observation_jouranl,amit_ltv_jouranl} to develop necessary conditions for mean square stabilization and observation of general nonlinear and linear time varying systems over uncertain channels. It is shown that global nonequilibrium dynamics of the nonlinear system play an important role in determining the minimum Quality of Service (QoS) of the erasure channel. The necessary condition for stabilization and observation of nonlinear systems were expressed in terms of the positive Lyapunov exponents of the nonlinear systems capturing the nonequilibrium dynamics and the statistics of channel uncertainty.

In this paper, we continue this line of research to provide a sufficient condition for mean square stabilization for a class of nonlinear systems in Lure form.
Deriving non-trivial sufficient condition for the stabilization of general nonlinear system over uncertain channels is a challenging problem. Hence, we focus on a particular class of nonlinear systems namely nonlinear systems in {\it Lure} form. A system in Lure form consist of feedback interconnection of LTI system and static nonlinearity. Systems in Lure form are widely studied in control system community because several systems in engineering application can be modeled as feedback interconnection of LTI system and static nonlinearity. Systematic analysis tools in the form of Positive Real Lemma (PRL) and Kalman-Yakubovich-Popov (KYP) Lemma exist for the synthesis and design of system in Lure form \cite{haddad_bernstein_DT,murat_kokotovic,ibrir_circle_criterion,Johansson20021557}. We make use of these powerful analysis methods in the development of the main results of this paper. The main contributions of the results derived in this paper are as follows. First, we discover a stochastic variant of PRL for Lure systems with multiplicative parametric uncertainty in the system matrices. This stochastic PRL is then used to provide a synthesis method for the design of observer based controller for the stabilization of nonlinear systems in Lure form over uncertain channels at input and outputs. We prove that the design of an observer-based controller for a general nonlinear systems satisfying some conditions on  the state feedback stabilization and observer design, interacting with the controller over packet-drop channels, enjoys the separation property i.e., the stabilizing controller and the observer can be designed independent of each other. The main result of this paper on the sufficiency condition has a interesting interpretation and helps us understand how the passivity property of the system can be traded off to account for channel uncertainty in feedback loop.

This paper is an extended version of the paper that appeared in the proceedings of American Control Conference, 2012 \cite{amit_lure_conference}. The paper includes important new additions compared to \cite{amit_lure_conference}. In particular, the separation principle for the design of observer-based controller for general nonlinear systems in the presence of channel uncertainty is new to this paper. This new result has allowed us to provide simpler and improved sufficient conditions for stability as compared to \cite{amit_lure_conference}.

The organization of the paper is as follows. In section \ref{prelim}, we provide some preliminaries and the stability definition used in this paper. The main results of this paper are proved in section \ref{main}. Simulation results are provided in section \ref{section_sim} followed by conclusions in section \ref{section_conclusion}.

\section{Preliminaries}
\label{prelim}
Consider the nonlinear system in Lur'e form with channel uncertainty at the inputs and outputs (refer to Fig. \ref{schematic}), described by following equations
\begin{align}
\label{nonlinear_sys}
x_{t+1} &= Ax_t - B\phi(y_t) + B \tilde u_t,\;\;\; y_t = Cx_t \\
\tilde u_t&=\gamma_t u_t,\;\;\;\tilde y_t = \xi_ty_t,
\end{align}
where, $x_t\in \mathbb{R}^N$, $\tilde u_t\in \mathbb{R}^M$, and $y_t\in \mathbb{R}^M$ are the state, control input, and output vector respectively. The matrices $B$ and $C$ are assumed to have full column rank and full row rank respectively. The random variables $\gamma_t$ and $\xi_t$ model the uncertainty at the input and output channels respectively.

We make the following assumptions on the system dynamics, channel uncertainties, and information structure between the input and output channels.
\begin{assumption}\label{assume_nonlinearity}
The nonlinearity $\phi(y) \in \mathbb{R}^M$ is globally Lipschitz, at least $C^1$ and monotonic non-decreasing function of $y \in \mathbb{R}^M$. Furthermore, this nonlinearity satisfies the following sector conditions
\begin{align}
\label{sector_control}
\phi(y)'\left(y - D_1 \phi(y)\right) > 0
\end{align}
\begin{align}
\label{sector_observer}
\left(\phi(y_1)-\phi(y_2)\right)'\left(\left(y_1 - y_2\right) - D_2 \left(\phi(y_1)-\phi(y_2)\right)\right) > 0
\end{align}
where $D_1+D_1' > 0$ and $D_2 + D_2' > 0$.
\end{assumption}
\begin{remark}
With no stability assumption on the system matrix $A$ it is possible to transform a nonlinearity in the feedback loop to satisfy the above conditions after appropriate loop transformation \cite{Khalil_book} provided the nonlinearity is globally Lipschitz and $C^1$. The matrices $D_1$ and $D_2$ are required to be positive definite to ensure that the nonlinearities remain in the positive quadrant to guarantee that the nonlinearity is passive in nature. This can be seen by rewriting \eqref{sector_control}, for $D_1'+D_1 > 0$ as follows,
\begin{align*}
2\phi(y)'y = \phi(y)'y + y'\phi(y) > \phi(y)'D_1\phi(y)+\phi(y)'D_1'\phi(y) = \phi(y)'(D_1'+D_1)\phi(y) > 0
\end{align*}
Thus the nonlinearity is strictly passive in nature if and only if there exists $D_1$ positive definite that satisfies \eqref{sector_control}. The argument for \eqref{sector_observer} is identical.
\end{remark}
\begin{assumption}\label{assume_uncertainty} We assume that the input channel uncertainty $\gamma_t\in\{0,1\}$ 
and output channel uncertainty $\xi_t\in\{0,1\}$. Both the channel uncertainties are assumed to be independent identically distributed (i.i.d) random variables with following statistics
\begin{eqnarray}
\textrm{Prob}\{\gamma_t = 1\} = p,\;\;\forall t,\quad \textrm{Prob}\{\xi_t = 1\} = q,\;\;\forall t.\label{input_opt_uncertainty}
\end{eqnarray}
To make the problem interesting, we assume that  $0<p,q<1$.
\end{assumption}

\begin{figure}[h]
\begin{center}
\vspace{-0.1in}
\includegraphics[width=20em]{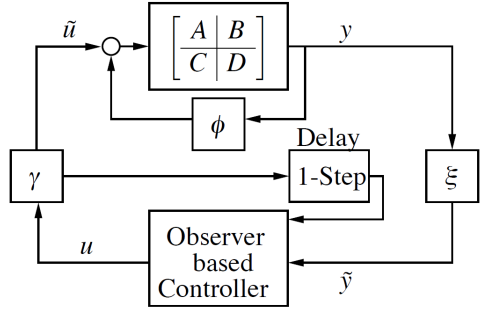}
\vspace{-0.3in}
\caption{Observer based controller for Lur'e system over uncertain channels}
\label{schematic}
\end{center}
\end{figure}

\begin{assumption}\label{TCP_protocol}
We assume that the observer receives the acknowledgment about the input  channel state with one-step delay (refer to Fig. \ref{schematic}). This acknowledgement structure is called as Transmission Control Protocol (TCP) \cite{networksystems_foundation_sastry,network_basar_tamar}.
\end{assumption}
\begin{remark}
We give a brief explanation for the Assumption \ref{TCP_protocol}. The wireless communication channel at the sensor and actuator channels is modeled as a packet erasure channel with an acknowledgement structure, which provides the sender with a binary acknowledgement of whether the packet was received or lost. The acknowledgement communication is assumed to be lossless. As the communication channels are of the binary erasure type and the ackenowledgement structure is binary, the communication channel has been modelled using a Bernoulli random variable. Such packet erasure channel models with an acknowledgement structure are utilized in modelling communication channels following the TCP protocol, and have been previously used in literature \cite{network_basar_tamar}. We outline the process of communication over the erasure channels in the feedback loop over one time step. Suppose at any given instant we have the observed state for $x_t$ given by $\hat{x}_t$. Using $\hat{x}_t$ the controller generates the input to be used $u_t = K(\hat{x}_t)$, which is further multiplied by $\gamma_t$ as the control is applied over the erasure channel. As the system generates $x_{t+1}$ using $\gamma_t$ and $u_t$, we obtain the output $y_t$. The output is communicated to the observer through an uncertain channel with multiplicative uncertainty $\xi_t$. Along with the output he channel also relays an acknowledgement ($\xi_{ack}$) for receipt of signal where $\xi_{ack} = 1$ indicates signal was received while $\xi_{ack} = 0$ indicates signal was lost. If the uncertain channel is a Bernoulli channel $\xi_t = 1 \Leftrightarrow \xi_{ack}=1$ and $\xi_t = 0 \Leftrightarrow \xi_{ack}=0$. Thus henceforth, with abuse of notation we will denote the acknowledgement (for both input and output channels) as the uncertainties $\gamma_t$ and $\xi_t$ themselves. Using the output $y_t$, uncertainty $\xi_t$, the control $u_t$ and the control channel uncertainty value $\gamma_t$ from the previous step, the observer generates the observation for $x_{t+1}$ given by the observed state $\hat{x}_{t+1}$. This is then used by the controller to generate the next control output $u_{t+1}$ which will be used to generate the next state $x_{t+2}$. Thus the TCP acknowledgement structure allows us to use the control uncertainty to generate the observed state in the following step.
\end{remark}
We next provide the definition of Quality of Service (QoS) for a channel.
\begin{definition}\label{qos}[QoS] For the channel with
multiplicative stochastic uncertainty $\zeta \in  W \subseteq \mathbb{R}$  with finite
mean $\mu$ and finite variance $\sigma^2$, the quality of service
of the channel is defined as
\[{\mathcal Q}=\frac{\mu^2}{\sigma^2}.\]
\end{definition}
\begin{remark} The definition of QoS is related to other performance measure in statistics as coefficient of variation or in signal processing to popular signal to noise ratio (SNR). The signal to noise ratio is defined as $SNR=\frac{\mu}{\sigma}$ (sometime defined as $SNR=\frac{\mu^2}{\sigma^2}$ to ensure positivity). Thus the QoS as defined above is a practical useful measure of performance. For input erasure channel (\ref{input_opt_uncertainty}) with non-erasure probability $p$, the QoS ${\mathcal Q}_{\gamma}=\frac{p}{1-p}$ and for output erasure channel (\ref{input_opt_uncertainty}) QoS is given by ${\cal Q}_{\xi} = \frac{q}{1-q}$.
\end{remark}
The notion of stability we adapt to analyze the feedback control system (\ref{schematic}) is mean square exponential (MSE) stability. We define this stability in the context of a random dynamical system of the form 
\begin{eqnarray}
x_{t+1}=S(x_t,\varsigma_t)\label{general_rds},
\end{eqnarray}
where, $x_t \in X\subseteq \mathbb{R}^N$, $\varsigma_t \in W \subseteq \mathbb{R}$ for $t \geq 0$, are i.i.d random variables with finite mean and second moment. The system mapping $S: X\times W\rightarrow X$ is assumed to be at least $C^1$ with respect to $x_t\in X$ and measurable w.r.t  $\varsigma_t$. We assume that $x=0$ is an equilibrium point i.e., $S(0,\varsigma_t)=0$. The following notion of stability can be defined for random dynamical systems(RDS) \cite{Hasminskii_book,applebaum}.
\begin{definition}[ Mean Square  Exponential (MSE) Stable]\label{EMSS_def} The solution $x=0$ is said to be MSE stable  for $x_{t+1} = S(x_t,\varsigma_t)$ if there exists a positive constants $M<\infty$ and $\beta<1$ such that
\[E_{\varsigma_0^t}\left [\parallel x_{t+1}\parallel^2 \right]\leq M\beta^t\|x_0\|^2,\;\; \forall t\geq 0\]
for almost all w.r.t. Lebesgue measure initial condition $x_0\in X$ where $E_{\varsigma_0^t}[\cdot]$ is the expectation taken over the sequence $\{\varsigma(0),\ldots,\varsigma_t\}$.
\end{definition}

\section{Main Results}
\label{main}

The first main result of this paper provides sufficient condition for the stabilization of Lure system  expressed in terms of statistics of channel uncertainties  and the solution of Riccati equation. The result also provides synthesis method for the design of observer-based controller robust to channels uncertainties.
\begin{theorem}
\label{combined_sufficiency_p}
Consider the observer-based controller design problem for the system in Lure form (\ref{nonlinear_sys}) (refer to Fig. \ref{schematic} for the schematic) satisfying Assumptions \ref{assume_nonlinearity} \ref{assume_uncertainty}, and \ref{TCP_protocol}. The feedback control system is mean square exponentially stable if there exists $P^* > 0$, $Q^* > 0$ such that following conditions are satisfied
\begin{align}
\label{control_p}
\Sigma_1 - B'P^*B &> \left(\frac{1}{1+{\mathcal Q}_{\gamma}}\right)\left(\Sigma_1 + B'P^*B\right)\\
\label{observer_q}
\Sigma_2 - CQ^*C' &> \left(\frac{1}{1+{\mathcal Q}_{\xi}}\right)\left(\Sigma_2 + CQ^*C'\right),
\end{align}
where, ${\cal Q}_{\gamma}=\frac{p}{1-p}$ is the QoS of the input channel and ${\cal Q}_{\xi}=\frac{q}{1-q}$ is the QoS of the output channel. The matrix $P^*$ and $Q^*$ satisfy following Riccati equations
\begin{align*}
P^* &= A_1'P^*A_1 - A_1P^*B\left(\Sigma_1 + B'P^*B\right)^{-1}B'PA_1 + C'\Sigma_1^{-1}C\\
Q^* &= A_2Q^*A_2' - A_2Q^*C'\left(\Sigma_2 + CQ^*C'\right)^{-1}CQ^*A_2' + B\Sigma_2^{-1}B',
\end{align*}
where, $\Sigma_1 = D_1 + D_1'$, $\Sigma_2 = D_2+D_2'$, $A_1 = A - B\Sigma_1^{-1}C$, $A_2 = A - B\Sigma_2^{-1}C$. Furthermore, the controller gain $K^*$ and observer gain $L^*$ are given by following expressions
\begin{align*}
K^* &= -(B'P^*B)^{-1}B'P^*A_1\\
L^* &= A_2Q^*C'(CQ^*C')^{-1}.
\end{align*}

\end{theorem}
We postpone the proof of this theorem till the end of this section and now provide some intuition behind the sufficiency condition for mean square exponential stability.

The conditions (\ref{control_p}) and (\ref{observer_q}) can be interpreted as generalizations of the positivity conditions  from deterministic PRL (i.e., $\Sigma_1 - B'P^*B > 0$ and $\Sigma_2 - CQ^*C' > 0$). Thus for the uncertain system we require $\Sigma_1 - B'P^*B$ and $\Sigma_2 - CQ^*C'$ to be strictly bounded below and this lower bound is a function of channel uncertainty. The closer these lower bounds are to zero the the amount of tolerable uncertainty decreases.
We notice from Eqs. (\ref{control_p}) and (\ref{observer_q}) that the sufficient conditions involving input and output channels uncertainty are decoupled. This implies that the separation principle applies for the design of observer-based controller for the system in Lure form with input and output channels uncertainties. The observer-based controller problem can be decomposed into two separate problems of design of full state feedback controller and observer design problems. The separation property is in fact the consequence of assumed TCP like acknowledgment structure (Assumption \ref{TCP_protocol}) \cite{network_basar_tamar,networksystems_foundation_sastry}. Equations (\ref{control_p}) and (\ref{observer_q}) then provides sufficient conditions for the mean square stabilization of Lure system with full state feedback control and for the observer error dynamics respectively.
We now outline the various steps involved in the proof of Theorem \ref{combined_sufficiency_p}.
\begin{enumerate}
\item In Theorem \ref{observer_condition_thm}, we prove results for the design of observer for system in Lure form. Theorem \ref{observer_condition_thm} provides bound on the minimum allowable erasure probability of the output channel to maintain mean square exponential stability of the observer error dynamics.
\item We provide the solution to the full state feedback stabilization problem with channel uncertainty at the input in Theorem \ref{controller_condition_thm}.
\item Finally using Assumption \ref{TCP_protocol}, we prove in Theorem \ref{separation_thm} that the observer-based controller design problem for the Lure system over uncertain channel enjoys the separation property.
\end{enumerate}

Lemma \ref{lyapunov_thm} provides sufficiency condition for the mean square exponential stability of  general stochastic dynamical systems.
\begin{lemma}
\label{lyapunov_thm}
The stochastic dynamical system (\ref{general_rds}) is exponentially mean square stable as given in Definition \ref{EMSS_def} if there exists a Lyapunov function $V_S(x_t)=x_t'Px_t$ for some matrix $P = P' >0$ and positive constants $c_1$, $c_2$ and $c_3$ such that
\begin{align}
\label{bound_condn_lure}
c_1\parallel x_t \parallel^2 \leq V_S(x_t) &\leq c_2\parallel x_t \parallel^2\\
\label{lyapunov_condition}
E_{\zeta_t}\left[V_S(x_{t+1})\right] - V_S(x_t) &< -c_3 \parallel x_t \parallel^2; \forall t\geq 0.
\end{align}
\end{lemma}

\begin{proof}
From (\ref{lyapunov_condition}) we get
\begin{align}
\label{lyap_condn_mod}
E_{\zeta_t}\left[V_S(x_{t+1})\right] < V_S(x_t) -c_3 \parallel x_t \parallel^2,\; \forall t\geq 0.
\end{align}
Substituting $V_S(x_t)$ for $\parallel x_t \parallel^2$ from (\ref{bound_condn_lure}) into (\ref{lyap_condn_mod}) we get
\begin{align}
E_{\zeta_t}\left[V_S(x_{t+1})\right] < \left(1 - \frac{c_3}{c_2}\right)V_S(x_t).
\end{align}
Let $1 - \frac{c_3}{c_2} := \beta_1<1$, since $c_2$ can be chosen greater than $c_3$. Taking expectation over $\zeta_{0}^{t}$ in the above equation we get
\begin{align*}
E_{\zeta_{0}^{t}}\left[V_S(x_{t+1})\right] < \beta_1E_{\zeta_{0}^{t-1}}\left[V_S(x_t)\right]
< \beta_1^2 E_{\zeta_{0}^{t-1}}\left[ V_S(x_{t-1})\right]<\beta_{1}^{t+1}V_S(x_0).
\end{align*}
Finally using the bounds on $V_S(x)$ from (\ref{bound_condn_lure}) in the above equation we get the desired result.
\end{proof}

We propose a observer design with linear gain and is similar to the circle criteria-based observer design proposed in \cite{ibrir_circle_criterion,murat_kokotovic}. The observer dynamics is assumed to be of the form:
\begin{align}
\label{observer_sys}
\hat{x}_{t+1} &= A\hat{x}_t - B\phi\left(\hat{y}_t\right) + \gamma_tBu_t + L\left(\tilde y_t-\tilde{\hat{y}}_t\right)\\
\hat{y}_t &= C\hat{x}_t,\quad \tilde{\hat{y}}_t = \xi_t\hat{y}_t.
\end{align}
This gives the error dynamics, $e_t:=x_t-\hat x_t$, to be
\begin{align}
\label{error_sys1}
e_{t+1} = \left(A-\xi_tLC\right)e_t - B\left(\phi(y_t)-\phi(\hat{y}_t)\right),\;\;\; w_t = Ce_t,
\end{align}
where, $w_t:=y_t-\hat y_t$.
\begin{remark}\label{remark_ouput_bernoulli} It is important to notice that because of the erasure channel uncertainty at the output channel it is possible to assume that the observer has access to channel erasure state, $\xi_t$. In particular, whenever the system output, $\tilde y_t$, is zero (non-zero) the channel erasure state can be assumed to be equal to zero (one).
\end{remark}
 Writing $\psi(t,w_t) := \left(\phi(y_t)-\phi(\hat{y}_t)\right)$ we can write the error dynamics as
\begin{align}
\label{error_sys2}
e_{t+1} = \left(A-\xi_tLC\right)e_t - B\psi(t,w_t),\;\;\;
w_t = Ce_t,
\end{align}
where it is clear that $\psi(t,w_t)$ satisfies the sector condition $\psi(t,w_t)'\left(w_t - D_2\psi(t,w_t)\right) > 0$ as given by (\ref{sector_observer}). Theorem \ref{observer_condition_thm} is the main result on observer design for system in Lure form (\ref{nonlinear_sys}).
\begin{theorem}
\label{observer_condition_thm}
Consider a nonlinear system in Lure form (\ref{nonlinear_sys}) satisfying Assumptions \ref{assume_nonlinearity}, \ref{assume_uncertainty}, and \ref{TCP_protocol} and the observer dynamics as given in (\ref{observer_sys}). Then the error dynamics (\ref{error_sys2}) is mean square exponentially stable if
\begin{align}
\left(\Sigma_2 -CQ^*C'\right)> \left(\frac{1}{1+{\mathcal Q}_{\xi}}\right)\left(\Sigma_2 + CQ^*C'\right),
\end{align}
where, $\Sigma_2 = D_2+D_2' > 0$ and $Q^* = (Q^*)' > 0$ satisfies the Riccati equation
\begin{align*}
Q^* = A_2Q^*A_2' - A_2Q^*C'(\Sigma_2 + CQ^*C')^{-1}CQ^*A_2' + B\Sigma_2^{-1}B',
\end{align*}
with $A_2 := A-B\Sigma_2^{-1}C$. Furthermore, the observer gain $L$ is given by,
\begin{align*}
L = A_2Q^*C'(CQ^*C')^{-1}.
\end{align*}
\end{theorem}

\begin{proof}
Consider the candidate Lyapunov function $V_t=e_t'P_oe_t$, where $P_o$ satisfies following equation.
\begin{align}
\label{observer_Riccati_uncertain_transform}
P_o =  R_o + E_{\xi_t}\left[ \mathcal{A}_o(\xi_t)'P_o\mathcal{A}_o(\xi_t)
+\left(\mathcal{A}_o(\xi_t)'P_oB - C'\right)(\Sigma_2-B'P_oB)^{-1}\left(B'P_o\mathcal{A}_o(\xi_t)-C\right)\right],
\end{align}
where, $\mathcal{A}_o(\xi_t) := A - \xi_tLC$ and $R_o >0$. Equation \ref{observer_Riccati_uncertain_transform} can be viewed as a stochastic variant of positive real Lemma Riccati equation \cite{haddad_bernstein_DT}.
Using (\ref{observer_Riccati_uncertain_transform}) and writing $\Delta V_t := E_{\xi_t}\left[V_{t+1}\right] - V_t$ we get
\begin{align}
\label{delta_V1}
\Delta V_t = &-e_t'R_oe_t - E_{\xi}\left[e_t' \left(\mathcal{A}_o(\xi)'P_oB - C'\right)\left(\Sigma_2 - B'P_oB\right)^{-1} \left(B'P_o\mathcal{A}_o(\xi) - C\right)e_t\right]\nonumber\\
 &-E_{\xi}\left[e_t'\mathcal{A}_o(\xi)'P_oB\psi(t,w_t)\right]-E_{\xi}\left[\psi(t,w_t)'B'P_o\mathcal{A}_o(\xi)e_t\right]+ E_{\xi}\left[\psi(t,w_t)'B'P_oB\psi(t,w_t)\right].
\end{align}
Add and subtract $\psi(t,w_t)'\Sigma_2\psi(t,w_t)$ and $2\psi(t,w_t)'w_t$ to (\ref{delta_V1}) to get
\begin{align}
\label{delta_V2}
\Delta V_t = &-e_t'R_oe_t - E_{\xi}\left[e_t' \left(\mathcal{A}_o(\xi)'P_oB - C'\right)\left(\Sigma_2 - B'P_oB\right)^{-1} \left(B'P_o\mathcal{A}_o(\xi) - C\right)e_t\right]\nonumber\\
 &-E_{\xi}\left[e_t'\left(\mathcal{A}_o(\xi)'P_oB - C'\right)\psi(t,w_t)\right]-E_{\xi}\left[\psi(t,w_t)'\left(C - B'P_o\mathcal{A}_o(\xi)\right)e_t\right]\nonumber\\
 &- E_{\xi}\left[\psi(t,w_t)'\left(\Sigma_2 - B'P_oB\right)\psi(t,w_t)\right] - 2\psi(t,w_t)'\left(w_t - D_2\psi(t,w_t)\right).
\end{align}
Using the algebraic manipulation given in \cite{haddad_bernstein_DT}, we express the above expression as a combination of negative definite functions as following,
\begin{align}
\Delta V_t = -e_t'R_oe_t - E_{\xi}\left[\nu(\xi)'\nu(\xi)\right]
 - 2\psi(t,w_t)'\left(w_t - D_2\psi(t,w_t)\right),
\end{align}
where, $\nu(\xi) = \left(\Sigma_2 - B'P_oB\right)^{-1/2}(B'P_o\mathcal{A}_o(\xi) - C)e_t + (\Sigma_2 - B'P_oB)^{1/2}\psi(t,w_t)$. Thus using the fact that $\psi(t,w_t)$ satisfies the sector condition we get
\begin{align*}
E_{\xi_t}\left[V_{t+1}\right] - V_t < -e_t'R_oe_t.
\end{align*}
Hence the asymptotic observer with erasure in sensor measurement is mean square exponentially stable. From (\cite{lancaster} Proposition 12.1.1) and the transformation $\mathcal{A}_{2o}:= \mathcal{A}_o - B\Sigma_2^{-1}C$, the equation (\ref{observer_Riccati_uncertain_transform}) can be written as
\begin{align*}
P_o > E_{\xi}\left[ \mathcal{A}_{2o}(\xi)'P_o\mathcal{A}_{2o}(\xi)+\mathcal{A}_{2o}(\xi)'P_oB(\Sigma_2-B'P_oB)^{-1}B'P_o\mathcal{A}_{2o}(\xi)\right]+C'\Sigma_2^{-1}C.
\end{align*}
This may then be written as
\begin{align}
\label{primal_ineq}
I > E_{\xi}\left[\tilde{\mathcal{A}}_2(\xi)'\tilde{\mathcal{A}}_2(\xi)\right],
\end{align}
where, $\tilde{\mathcal{A}}_2(\xi) = (P_o^{-1} - B\Sigma_2^{-1}B')^{-\frac{1}{2}}\left(\mathcal{A}_{2o}(\xi) - B\Sigma_2^{-1}C\right)(P_o - C'\Sigma_2^{-1}C)^{-\frac{1}{2}}$. We know that (\ref{primal_ineq}) is true if and only if
\begin{align}
\label{dual_ineq}
I > E_{\xi}\left[\tilde{\mathcal{A}}_2(\xi)\tilde{\mathcal{A}}_2(\xi)'\right].
\end{align}
Now defining $Q_o:= P_o^{-1}$ and expanding (\ref{dual_ineq}) we get
\begin{align}
\label{dual_Riccati_observer}
Q_o - B\Sigma_2^{-1}B' > E_{\xi}\left[\mathcal{A}_{2o}(\xi)(Q_o^{-1} - C'\Sigma_2^{-1}C)^{-1}\mathcal{A}_{2o}(\xi)'\right].
\end{align}
Minimizing trace of the right hand side (RHS) in above equation, with respect to $L$ we get
\begin{align*}
L = A_2(Q_o^{-1} - C'\Sigma_2^{-1}C)^{-1}C'\left(C(Q_o^{-1} - C'\Sigma^{-1}C)^{-1}C'\right)^{-1},
\end{align*}
where, $A_2 = A - B\Sigma_2^{-1}C$. Simple matrix computation gives us $ A_2(Q_o^{-1} - C'\Sigma_2^{-1}C)^{-1}C'=A_2Q_oC'(\Sigma_2 - CQ_oC')^{-1}\Sigma_2$ and $C(Q_o^{-1} - C'\Sigma_2^{-1}C)^{-1}C' =  CQ_oC'(\Sigma_2 - CQ_oC')^{-1}\Sigma_2$. Applying these matrix simplifications to the gain $L$ we get
\[ L = A_2Q_oC'(CQ_oC')^{-1}. \]
Substituting this structure of $L$ in (\ref{dual_Riccati_observer}) we get
\begin{align}
\label{optimal_Riccati}
Q_o > A_2Q_oA_2' - A_2Q_oC'\left[q(CQ_oC')^{-1}- (1-q)(\Sigma_2 - CQ_oC')^{-1}\right]CQ_oA_2'+ B\Sigma^{-1}B'.
\end{align}
We now wish to design $Q_o = Q^*$ that will satisfy the above equation. Now suppose $Q^*$ satisfies the minimum covariance like Riccati equation given by
\begin{align*}
Q^* = A_2Q^*A_2' - A_2Q^*C'(\Sigma_2 + CQ^*C')^{-1}CQ^*A_2' + B\Sigma_2^{-1}B',
\end{align*}
then $Q^*$ satisfies (\ref{optimal_Riccati}) if
\begin{align*}
q(CQ^*C')^{-1} - (1-q)(\Sigma_2 - CQ^*C')^{-1} > (\Sigma_2 + CQ^*C')^{-1}.
\end{align*}
Thus the observer error dynamics (\ref{error_sys2}) is exponentially mean square stable if
\begin{align}
\Sigma_2 - CQ^*C' > \left(1-q\right)\left(\Sigma_2 + CQ^*C'\right).
\end{align}
Substituting ${\cal Q}_{\xi} = \frac{q}{1-q}$ proves the result.
\end{proof}
Theorem \ref{controller_condition_thm}, provide results for the design of full state feedback controller, $u_t=K x_t$ for system (\ref{nonlinear_sys}) in Lure form.
\begin{theorem}
\label{controller_condition_thm}
Consider the system (\ref{nonlinear_sys}) in Lure form satisfying Assumptions \ref{assume_nonlinearity}, \ref{assume_uncertainty}, and \ref{TCP_protocol}. Let $u_t=Kx_t$ be the linear full state feedback controller, then the state dynamics is mean square exponentially stable if
\begin{align}
\left(\Sigma_1 -B'P^*B\right)> \left(\frac{1}{1+{\mathcal Q}}\right)\left(\Sigma_1 + B'P^*B\right),
\end{align}
 where, $\Sigma_1 = D_1+D_1' > 0$. The matrix $P^* = (P^*)' > 0$ satisfies the Riccati equation
\begin{align*}
P^* = A_1'P^*A_1 - A_1'P^*B(\Sigma_1 + B'P^*B)^{-1}B'P^*A_1 + C'\Sigma_1^{-1}C,
\end{align*}
where, $A_1 := A-B\Sigma_1^{-1}C$. Furthermore, the controller gain is given by
\begin{align*}
K = -(B'P^*B)^{-1}B'P^*A_1.
\end{align*}
\end{theorem}

\begin{proof}
Consider the candidate Lyapunov function $V_t=x_t'P_cx_t$. Then, following the proof of Theorem \ref{observer_condition_thm}, the state feedback controller with erasure in actuator is mean square exponentially stable if $P_c$ satisfies
\begin{align}
\label{controller_Riccati_uncertain_transform}
P_c =  R_c + E_{\gamma_t}\left[ \mathcal{A}_c(\gamma_t)'P_c\mathcal{A}_c(\gamma_t)
+\left(\mathcal{A}_c(\gamma_t)'P_cB - C'\right)(\Sigma_1-B'P_cB)^{-1}\left(B'P_c\mathcal{A}_c(\gamma_t)-C\right)\right],
\end{align}
where, $\mathcal{A}_c(\gamma_t) := A + \gamma_tBK$ and $R_c >0$. From \cite{lancaster} and the transformation $\mathcal{A}_{1c}:= \mathcal{A}_c - B\Sigma_1^{-1}C$, the equation (\ref{observer_Riccati_uncertain_transform}) can be written as
\begin{align}%
\label{controller_Riccati_uncertain}
P_c > E_{\gamma}\left[ \mathcal{A}_{1c}(\gamma)'P_c\mathcal{A}_{1c}(\gamma)+\mathcal{A}_{1c}(\gamma)'P_cB(\Sigma_1-B'P_cB)^{-1}B'P_c\mathcal{A}_{1c}(\gamma)\right]+C'\Sigma_1^{-1}C.
\end{align}
Minimizing trace of RHS in above equation, with respect to $K$ we get
\begin{align*}
K = -\left(B'(P_c^{-1} - B\Sigma_1^{-1}B')^{-1}B\right)^{-1}B'(P_c^{-1} - B\Sigma_1^{-1}B')^{-1}A_1,
\end{align*}
where, $A_1 = A - B\Sigma_1^{-1}C$. Simple matrix computation gives us $ A_1(P_c^{-1} - B\Sigma_1^{-1}B')^{-1}B=A_1P_cB(\Sigma_1 - B'P_cB)^{-1}\Sigma_1$ and $B'(P_c^{-1} - B\Sigma_1^{-1}B')^{-1}B =  B'P_cB(\Sigma_1 - B'P_cB)^{-1}\Sigma_1$. Applying these matrix simplifications to the gain $K$ we get
\[ K = -(B'P_cB)^{-1}B'P_cA_1. \]
Substituting this structure of $K$ in (\ref{controller_Riccati_uncertain}), we get
\begin{align}
\label{optimal_Riccati_control}
P_c> A_1'P_cA_1 - A_1'P_cB\left[p(B'P_cB)^{-1}- (1-p)(\Sigma_1 - B'P_cB)^{-1}\right]B'P_cA_1+ C'\Sigma_1^{-1}C.
\end{align}
We now wish to design $P_c = P^*$ that will satisfy the above equation. Now suppose $P^*$ satisfies the minimum covariance like Riccati equation given by
\begin{align*}
P^* = A_1P^*A_1' - A_1P^*B(\Sigma_1 + B'P^*B)^{-1}B'P^*A_1 + C'\Sigma_1^{-1}C,
\end{align*}
then $P^*$ satisfies (\ref{optimal_Riccati_control}) if
\begin{align*}
p(B'P^*B)^{-1} - (1-p)(\Sigma_1 - B'P^*B)^{-1} > (\Sigma_1 + B'P^*B)^{-1}.
\end{align*}
Thus the controller dynamics (\ref{nonlinear_sys}) is exponentially mean square stable if
\begin{align}
\Sigma_1 - B'P^*B > \left(1 - p\right)\left(\Sigma_1 + B'P^*B\right).
\end{align}
Substituting ${\cal Q}_{\gamma} = \frac{p}{1-p}$ proves the result.
\end{proof}
Our next theorem in on separation principle. We prove that the problem of observer-based controller design for nonlinear systems with uncertainties at the input and output channels can be decomposed into two separate problems of designing a full state feedback controller and a observer design problem. We prove the theorem on separation principle for more general nonlinear systems with channel uncertainties at the input and output.
\begin{eqnarray}
x_{t+1} &=& f(x_t, \gamma_t,u_t),\;\;\;
y_t =h(x_t),\;\;\;\tilde y_t= \xi_t y_t\nonumber\\
\hat x_{t+1} &=& g(\hat x_t,  \gamma_t,u_t, \xi_t,h(x_t))\label{coupled_sys},
\end{eqnarray}
where, $x_t \in X \subseteq \mathbb{R}^N$, $y_t \in Y \subseteq \mathbb{R}^M$ and $u_t \in U \subseteq \mathbb{R}^M$ are the state, output and input respectively. $\hat x_t\in \mathbb{R}^N$ is the observer state. $\gamma_t\in  W \subseteq \mathbb{R}$ and $\xi_t\in \{0,1\}$ are  assumed to be i.i.d random variables modeling the uncertainty at the input and output channels respectively. We notice that the observer dynamics $g$ in (\ref{coupled_sys}) is assumed to be the function of function of input channel erasure state $\gamma_t$. This is because of the assumed acknowledgement structure in the form of TCP (Assumption \ref{TCP_protocol}). We have already provided Lyapunov based conditions for mean square stability of the observer error dynamics and the stability under full state feedback for the system in Lure form (Theorem \ref{controller_condition_thm} and \ref{observer_condition_thm}). We will now prove the separation theorem for a general nonlinear system with uncertainty in control and observation. We will make the following assumptions regarding the nonlinear system and its Lyapunov based control and observation. It should be noted that these assumptions (Assumptions \ref{assume_Lyapunov} and \ref{assume_bounds}) are satisfied by the Lure system.
 We now make following assumption on system (\ref{coupled_sys}).
\begin{assumption}\label{assume_Lyapunov} For system (\ref{coupled_sys}), let $u_t=k(x_t)$ be the full state feedback control input, we assume that there exist Lyapunov functions $V_1(x_t)$ and $V_2(e_t)$, with $e_t:=x_t-\hat x_t$, and positive constants $\bar c_1$, $\bar c_2$, $\bar c_3$, $d_1$, $d_2$ and $d_3$ such that following conditions are satisfied.
\begin{align}
\bar c_1||x_t||^2 \leq V_1(x_t) \leq \bar c_2||x_t||^2,\;\;\;
E_{\gamma_t}[V_1(x_{t+1})] - V_1(x_t)\leq -\bar c_3||x_t||^2 \label{v1_b}, \\
d_1||e_t||^2 \leq V_2(e_t) \leq d_2||e_t||^2,\;\;\;
E_{\gamma_t,\xi_t}[V_2( e_{t+1})] - V_2(e_t)\leq -d_3||e_t||^2\label{v2_b}.
\end{align}
\end{assumption}

\begin{assumption}\label{assume_bounds}
 We assume that there exists positive constants $L_3$, $L_4$, $L_5$, $L_6$, and $L_7$ such that $||\frac{\partial f}{\partial x}|| \leq L_3$, $||\frac{\partial f}{\partial u}|| \leq L_4$,  $||k(x_t) - k(\hat x_t)|| \leq L_5||e_t||$, $||\frac{\partial V_1}{\partial x}(x)|| \leq L_6||x||$, and $||\frac{\partial V_2}{\partial e}(e)|| \leq L_7||e||$ .
\end{assumption}
\begin{remark} The results on separation principle for deterministic nonlinear systems exist \cite{vidyasagar_separation}. Our results in Theorem \ref{separation_thm} extends these results for nonlinear systems with input and output channel uncertainty. The results in Theorem \ref{separation_thm} can be considered as one of the contribution of this paper. Furthermore, separation theorem is proved for more general uncertainty and this will allow us to use the results of Theorem \ref{separation_thm} in the proof of second main result of this paper (Theorem \ref{separation_thm}).
\end{remark}
\begin{remark}
The existence of Lyapunov functions satisfying conditions (\ref{v1_b}) and (\ref{v2_b}) in Assumption \ref{assume_Lyapunov} combined with the results from Lyapunov based Theorem \ref{lyapunov_thm}
 ensure that state dynamics and observer error dynamics for system (\ref{coupled_sys}) is mean square exponentially stable. In particular, it follows that state dynamics, with full state feedback, and observer  error dynamics for (\ref{coupled_sys}) satisfies following stability conditions.
\[ E_{\gamma_0^t}[||x_{t+1}||^2] \leq M_1\bar \beta_1^{t+1}||x_0||^2,\quad \quad  E_{\gamma_0^t,\xi_0^t}[||e_{t+1}||^2] \leq M_2\bar \beta_2^{t+1}||e_0||^2,\]
for some positive constants $M_1<\infty$, $M_2<\infty$, $\bar \beta_1 := 1- \frac{\bar c_3}{\bar c_2}<1$, and $\bar \beta_2:= 1- \frac{d_3}{d_2}<1$.
\end{remark}

Theorem \ref{separation_thm} is our main result on principle of separation.
\begin{theorem}
\label{separation_thm} Consider the observer-based controller design problem for system (\ref{coupled_sys}) satisfying Assumptions \ref{assume_Lyapunov} and \ref{assume_bounds}. Then the state dynamics of system (\ref{coupled_sys}) using estimated state, $\hat x_t$, for  the feedback control input (i.e., $u_t=k(\hat x_t)$) is mean square exponentially stable.
\end{theorem}

\begin{proof}
For $u_t = k(\hat x_t)$ we obtain
\begin{align}
&E_{\gamma_t,\xi_t}\left[V_1(f(x_t,\gamma_t,k(\hat x_t)))\right] - V_1(x_t)\nonumber\\
&=E_{\gamma_t}\left[V_1(f(x_t,\gamma_t,k(x_t)))\right] - V_1(x_t) + E_{\gamma_t,\xi_t}\left[V_1(f(x_t,\gamma_t,k(\hat x_t)))\right] - E_{\gamma_t}\left[V_1(f(x_t,\gamma_t,k(x_t)))\right]\nonumber\\
&=E_{\gamma_t}\left[V_1(f(x_t,\gamma_t,k(x_t)))\right] - V_1(x_t) + E_{\gamma_t,\xi_t}\left[V_1(f(x_t,\gamma_t,k(\hat x_t))) - V_1(f(x_t,\gamma_t,k(x_t)))\right] \label{sep_eq_1}
\end{align}
Applying Mean Value Theorem and equation (\ref{v1_b}) to (\ref{sep_eq_1}) and using Assumption \ref{assume_bounds} we obtain
\begin{align}
E_{\gamma_t,\xi_t}\left[V_1(f(x_t,\gamma_t,k(\hat x_t)))\right] - V_1(x_t)
&\leq -c_3||x_t||^2\nonumber\\ 
&\;+ E_{\gamma_t,\xi_t}\left[\bigg\|\frac{\partial V_1}{\partial f}(f(x_t,\gamma_t,r))\bigg\| \bigg\|\frac{\partial f}{\partial k}(x_t,\gamma_t,r)\bigg\| \big\|k(\hat x_t) - k(x_t)\big\|\right]\nonumber\\
&\leq -c_3||x_t||^2 + E_{\gamma_t,\xi_t}\left[L_4L_5L_6\big\|f(x_t,\gamma_t,r)\big\| ||\hat{x}_t - x_t||\right],
\end{align}
where, $r = sk(x_t) + (1-s)k(\hat x_t)$ for some $0 \leq s \leq 1$. Thus, writing $c_4 := L_4L_5L_6$, and $\beta_1 := 1 -\frac{c_3}{c_2}$,
\begin{align}
E_{\gamma_t,\xi_t}\left[V_1(f(x_t,\gamma_t,k(\hat x_t)))\right] &\leq \beta_1V_1(x_t) \nonumber\\
&\;+ c_4E_{\gamma_t,\xi_t}\left[||f(x_t,\gamma_t,r) - f(x_t,\gamma_t,k(x_t)) + f(x_t,\gamma_t,k(x_t))||\right]||x_t - \hat{x}_t||\nonumber\\
&\leq \bar{\beta}_1V_1(x_t) + c_4E_{\gamma_t,\xi_t}\left[||f(x_t,\gamma_t,r) - f(x_t,\gamma_t,k(x_t))||\right]||x_t - \hat{x}_t||\nonumber\\
&\quad + c_4E_{\gamma_t,\xi_t}\left[||f(x_t,\gamma_t,k(x_t))||\right]||x_t - \hat{x}_t||\nonumber\\
&\leq \bar{\beta}_1V_1(x_t) + c_4L_4L_5||x_t - \hat{x}_t||^2 + c_4E_{\gamma_t,\xi_t}\left[||f(x_t,\gamma_t,k(x_t))||^2\right]^{\frac{1}{2}}||x_t - \hat{x}_t||\nonumber\\
&\leq \bar{\beta}_1V_1(x_t) + c_4L_4L_5||x_t - \hat{x}_t||^2 + \frac{c_4}{\sqrt{\bar{c}_1}}\left(\bar{\beta}_1V_1(x_t)\right)^{\frac{1}{2}}||x_t - \hat{x}_t||.
\end{align}
Define $c_5 := c_4L_4L_5$ and $c_6 := \frac{c_4}{\sqrt{\bar{c}_1}}$. Since $\bar{\beta}_1 < 1$ there exists $\delta_1 > 0$ such that $\beta_3 := (1+\delta_1)\bar{\beta}_1 < 1$. Then if $(\bar{\beta}_1V_1(x_t))^{\frac{1}{2}} \geq \frac{c_6}{\delta_1}||x_t - \hat{x}_t||$, we obtain
\begin{align}
\label{condn1}
E_{\gamma_t,\xi_t}\left[V_1(f(x_t,\gamma_t,k(\hat x_t)))\right] &\leq \beta_3 V_1(x_t) + c_5||x_t - \hat{x}_t||^2.
\end{align}
In case $(\bar{\beta}_1V_1(x_t))^{\frac{1}{2}} \leq \frac{c_6}{\delta_1}||x_t - \hat{x}_t||$, we obtain
\begin{align}
\label{condn2}
E_{\gamma_t,\xi_t}\left[V_1(f(x_t,\gamma_t,k(\hat x_t)))\right] &\leq \bar{\beta}_1 V_1(x_t) + c_7||x_t - \hat{x}_t||^2,
\end{align}
where $c_7 := c_5 + \frac{c_6^2}{\delta_1}$. Thus taking the supremum over both conditions (\ref{condn1}) (\ref{condn2}) we obtain
\begin{align}\label{1stepeqn}
E_{\gamma_t,\xi_t}\left[V_1(f(x_t,\gamma_t,k(\hat x_t)))\right] &\leq \beta_3 V_1(x_t) + c_7||x_t - \hat{x}_t||^2.
\end{align}
Taking expectation over the uncertainty sequences $\{\xi_0,\ldots,\xi_t\}$, and $\{\gamma_0,\ldots,\gamma_t\}$, and recursively applying Eq. \eqref{1stepeqn} we obtain
\begin{align}
E_{\gamma_0^t,\xi_0^t}\left[V_1(f(x_t,\gamma_t,k(\hat{x}_t)))\right] &\leq \beta_3^{t+1} V_1(x_0) + c_7\sum_{i=1}^tE_{\xi_0^{i-1}}\left[||x_i - \hat{x}_i||^2\right]\beta_3^{t-i} + c_7 ||x_0 - \hat{x}_0||^2\beta_3^t\nonumber\\
&\leq \beta_3^{t+1} V_1(x_0) + c_7\sum_{i=0}^t\bar{\beta}_2^i\beta_3^{t-i} ||x_0 - \hat{x}_0||^2\nonumber\\
&\leq \beta_3^{t+1} V_1(x_0) + c_8\beta_4^{t+1}||x_0 - \hat{x}_0||^2,
\end{align}
where, $\beta_4 = \max(\beta_3,\beta_2)$, $c_8 := \frac{c_7\beta_4}{\beta_4 - \min(\beta_2,\beta_3)}$. Applying (\ref{v1_b}) we get 
\begin{align}
E_{\gamma_0^t,\xi_0^t}\left[||x_{t+1}||^2\right] &\leq \frac{\bar{c}_2}{\bar{c}_1}\beta_3^{t+1} ||x_0||^2 + \frac{c_8}{\bar{c}_1}\beta_4^{t+1}||x_0 - \hat{x}_0||^2,
\end{align}
We can now use the inequality $||x_0 - \hat{x}_0||^2 \leq 2(||x_0||^2 + ||\hat{x}_0||^2)$, along with (\ref{v2_b}) to get the desired result that the coupled system (\ref{coupled_sys}) is mean square exponentially stable.
\end{proof}
We are now ready to provide the proof of Theorem \ref{combined_sufficiency_p}.

\begin{proof}[Proof of Theorem \ref{combined_sufficiency_p}] It is easy to verify that the system in Lure form (\ref{nonlinear_sys}) satisfies the Assumption \ref{assume_bounds}. Furthermore, using the results from Theorems \ref{observer_condition_thm} and \ref{controller_condition_thm} it follows that there exist quadratic Lyapunov functions satisfying Assumption \ref{assume_Lyapunov}. Hence, results of Theorem \ref{separation_thm} are applicable to the observer based controller problem for (\ref{nonlinear_sys}) given controller and observer as designed in Theorems \ref{controller_condition_thm} and \ref{observer_condition_thm}. The proof then follows by combining the results of Theorems \ref{observer_condition_thm}, \ref{controller_condition_thm}, and \ref{separation_thm}.
\end{proof}
\subsection{Control of Lure system over uncertain channels}

In the main result of this paper we restricted the uncertain communication channels to packet-erasure channel with acknowledgement, modeled as Bernoulli random variables. We now provide results for state feedback control under general uncertainty model at the actuator. Consider system given by equation (\ref{nonlinear_sys}) satisfying Assumptions \ref{assume_nonlinearity} with full state observation and state feedback control. Suppose the uncertainty $\gamma_t$ is such that
\begin{assumption}\label{assumption_uncertain_general}
The random variables $\gamma_t$ are i.i.d. with statistics given by
\begin{align}\label{gam_stat}
E\left[\gamma_t\right] = \mu,\quad E\left[(\gamma_t - \mu)^2\right] = \sigma^2
\end{align}
\end{assumption}
We can then obtain the following theorem
\begin{theorem}
\label{controller_condition_thm_gen}
Consider the system (\ref{nonlinear_sys}) in Lure form with full state observation and state feedbck control. Suppose it satisfies Assumptions \ref{assume_nonlinearity} and \ref{assumption_uncertain_general}. Let $u_t=Kx_t$ be the linear full state feedback controller, then the state dynamics is mean square exponentially stable if
\begin{align}
\left(\Sigma_1 -B'P^*B\right)> \left(\frac{1}{1+{\mathcal Q}}\right)\left(\Sigma_1 + B'P^*B\right),
\end{align}
 where, $\Sigma_1 = D_1+D_1' > 0$. The matrix $P^* = (P^*)' > 0$ satisfies the Riccati equation
\begin{align*}
P^* = A_1'P^*A_1 - A_1'P^*B(\Sigma_1 + B'P^*B)^{-1}B'P^*A_1 + C'\Sigma_1^{-1}C,
\end{align*}
where, $A_1 := A-B\Sigma_1^{-1}C$. Furthermore, the controller gain is given by
\begin{align*}
K = -\frac{\mu}{\mu^2 + \sigma^2}(B'P^*B)^{-1}B'P^*A_1.
\end{align*}
\end{theorem}

\begin{proof}
Consider the candidate Lyapunov function $V_t=x_t'P_cx_t$.
Then, following the proof of Theorem \ref{controller_condition_thm}, the state feedback controller with erasure in actuator is mean square exponentially stable if $P_c$ satisfies
\begin{align}%
\label{controller_Riccati_uncertain_gen}
P_c > E_{\gamma}\left[ \mathcal{A}_{1c}(\gamma)'P_c\mathcal{A}_{1c}(\gamma)+\mathcal{A}_{1c}(\gamma)'P_cB(\Sigma_1-B'P_cB)^{-1}B'P_c\mathcal{A}_{1c}(\gamma)\right]+C'\Sigma_1^{-1}C.
\end{align}
where, $\mathcal{A}_c(\gamma_t) := A + \gamma_tBK - B\Sigma_1^{-1}C$. Minimizing trace of RHS in above equation, with respect to $K$ and simplyfying the expression as in Theorem \ref{controller_condition_thm} we obtain
\begin{align}
\label{gain_structure_gen}
K = -\frac{\mu}{\mu^2+\sigma^2}(B'P_cB)^{-1}B'P_cA_1,
\end{align}
where, $A_1 = A - B\Sigma_1^{-1}C$. 
Substituting structure of $K$ from \eqref{gain_structure_gen} in (\ref{controller_Riccati_uncertain_gen}), we obtain,
\begin{align}
\label{optimal_Riccati_control_gen}
P_c> A_1'P_cA_1 - A_1'P_cB\left[\frac{\mu^2}{\mu^2+\sigma^2}(B'P_cB)^{-1}- \frac{\sigma^2}{\mu^2+\sigma^2}(\Sigma_1 - B'P_cB)^{-1}\right]B'P_cA_1+ C'\Sigma_1^{-1}C.
\end{align}
We now wish to give design a $P^*$ that will satisfy the above equation. Now suppose $P^*$ satisfies the minimum covariance like Riccati equation given by
\begin{align*}
P^* = A_1P^*A_1' - A_1P^*B(\Sigma_1 + B'P^*B)^{-1}B'P^*A_1 + C'\Sigma_1^{-1}C,
\end{align*}
then $P^*$ satisfies (\ref{optimal_Riccati_control_gen}) if
\begin{align*}
\frac{\mu^2}{\mu^2+\sigma^2}(B'P^*B)^{-1} - \frac{\sigma^2}{\mu^2+\sigma^2}(\Sigma_1 - B'P^*B)^{-1} > (\Sigma_1 + B'P^*B)^{-1}.
\end{align*}
Thus the state feedback controlled dynamics (\ref{nonlinear_sys}) with general uncertainty in actuation is exponentially mean square stable if
\begin{align}
\Sigma_1 - B'P^*B > \left(\frac{\sigma^2}{\mu^2 + \sigma^2}\right)\left(\Sigma_1 + B'P^*B\right).
\end{align}
Substituting ${\cal Q}_{\gamma} = \frac{\mu^2}{\sigma^2}$ proves the result.
\end{proof}

\section{Simulation}\label{section_sim}
For simulation, we consider a discrete time system given by,
\begin{align}
x_{t+1} = \left[
\begin{array}{ccc}
0 & 1 & 0 \\
0 & 0 & 1 \\
4.6 & -3.5 & -1
\end{array}
\right]x_t- 
\left[
\begin{array}{c}
1 \\
0 \\
0
\end{array}
\right]\phi(y_t)
+\left[
\begin{array}{c}
1 \\
0 \\
0
\end{array}
\right]\tilde{u}_t
;\quad
y_t = \left[
\begin{array}{ccc}
1 & 0 & 0
\end{array}
\right]x_t.
\end{align}
The nonlinearity $\phi_I(y)$ is given by
\begin{align}
\phi(y) = \left\{
\begin{array}{cc}
0.1 y, & ||y|| < 1\\
13.6y - 13.5\text{sgn}(y), & 1 < ||y|| \leq 3 \\
4.6y + 13.5\text{sgn}(y), & 3 < ||y||
\end{array}
\right. .
\end{align}
This discrete time system demonstrates chaotic dynamics as shown in Figure (\ref{fig:dyn0}).
\begin{figure}[ht!]
\begin{center}
\includegraphics[width=4in]{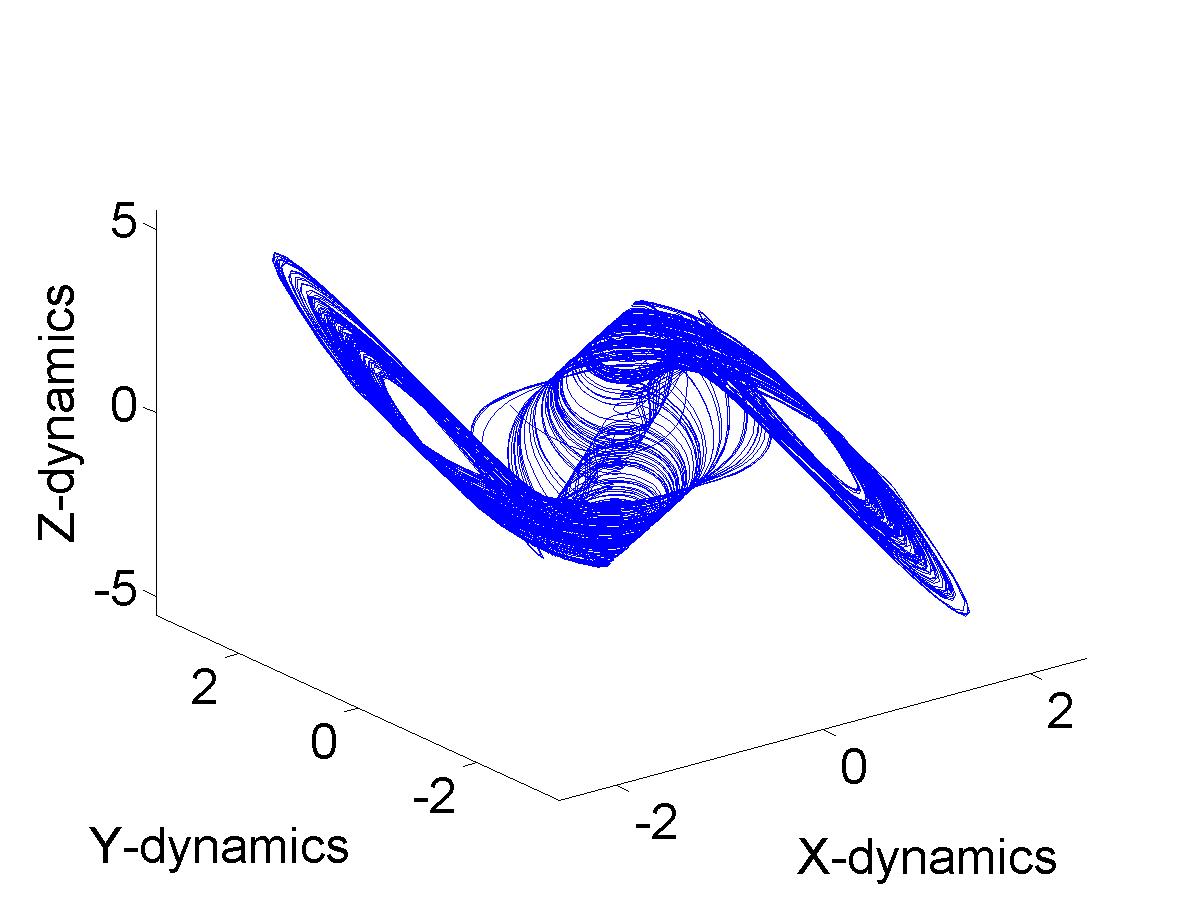}
\caption{State dynamics in 3 dimensions}
\label{fig:dyn0}
\end{center}
\end{figure}
We now implement our observer-based controller design on the above discrete time system. The uncertainty at the input and output are assumed to be Bernoulli random variable. From our sufficiency condition we get $p_c = 0.3441$ and $q_c = 0.3441$. We choose the probability of erasure for our simulation to be $p = 0.35$ and $q = 0.35$. In Figs. (\ref{fig:dyn}a) and (\ref{fig:dyn}b) we plot the observer-based controlled state and observer error dynamics. We see that they both decay to zero thereby verifying the sufficient condition of the main result. The state and observer error plots represent the outcome of average values of states and observer error over $50$ independent realizations of random variables $\xi$ and $\gamma$.

\begin{figure}[ht!]
\begin{center}
\mbox{
\hspace{-0.05in}
\subfigure[]{\scalebox{0.15}{\includegraphics{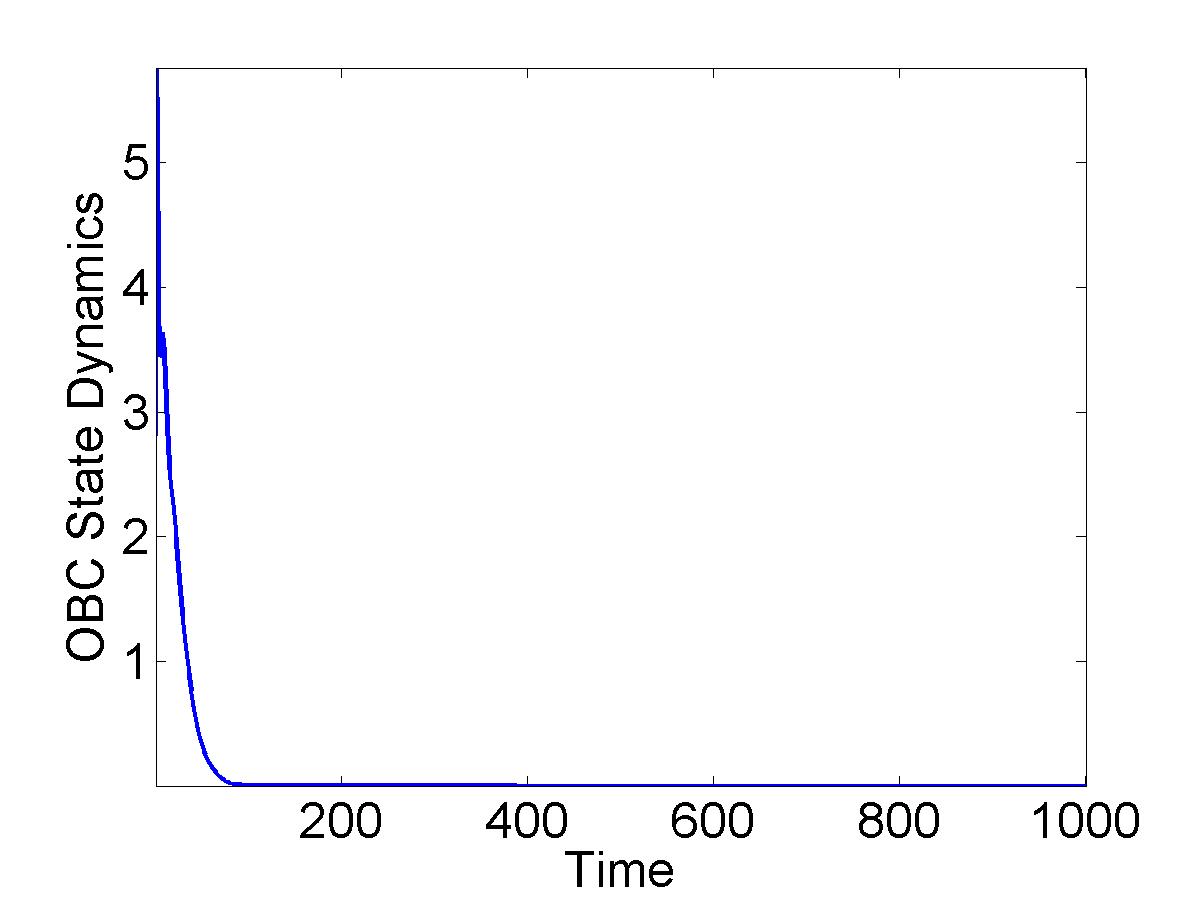}}}
\hspace{-0.05in}
\subfigure[]{\scalebox{0.15}{\includegraphics{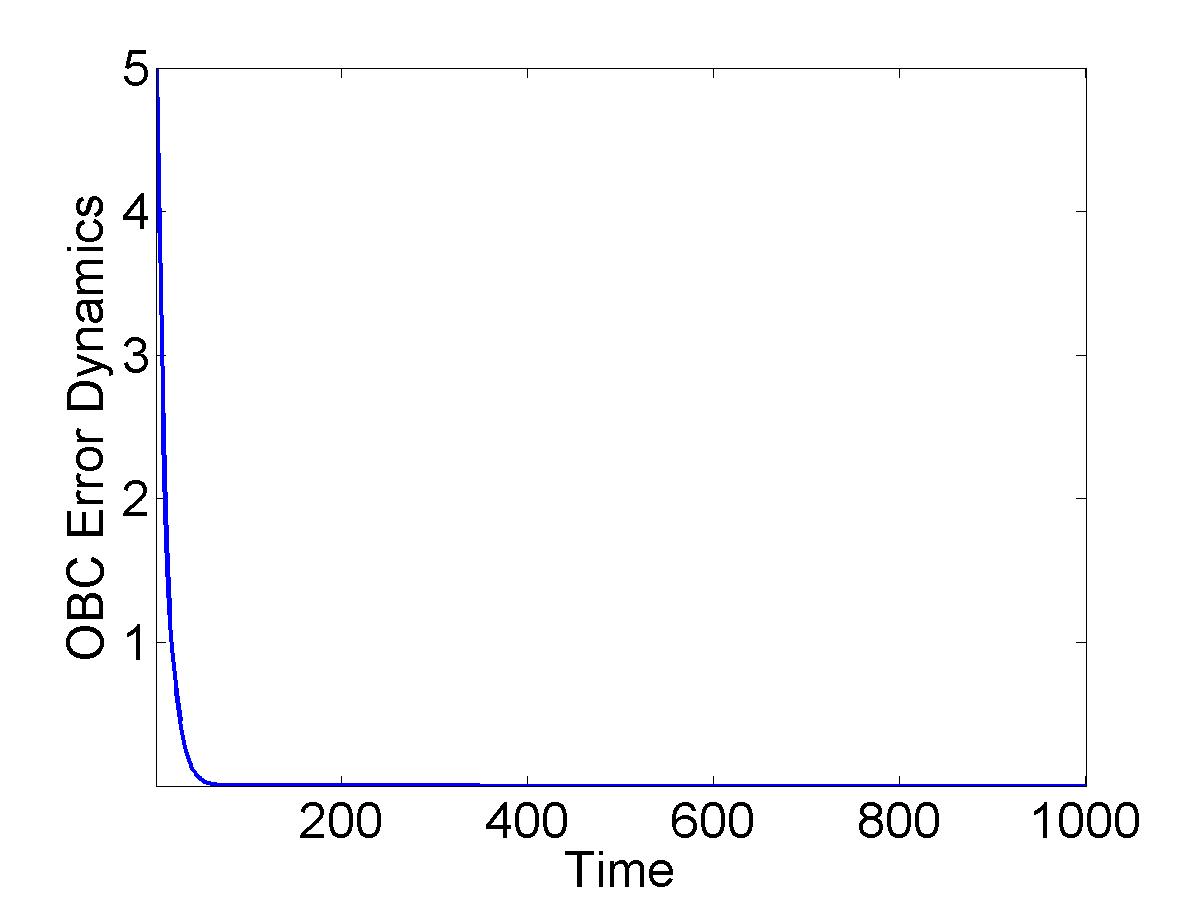}}} }
\caption{(a) Dynamics of controlled state, (b) Dynamics of observer error}
\label{fig:dyn}
\end{center}
\end{figure}
\begin{figure}[htb!]
\begin{center}
\mbox{
\hspace{-0.05in}
\subfigure[]{\scalebox{0.15}{\includegraphics{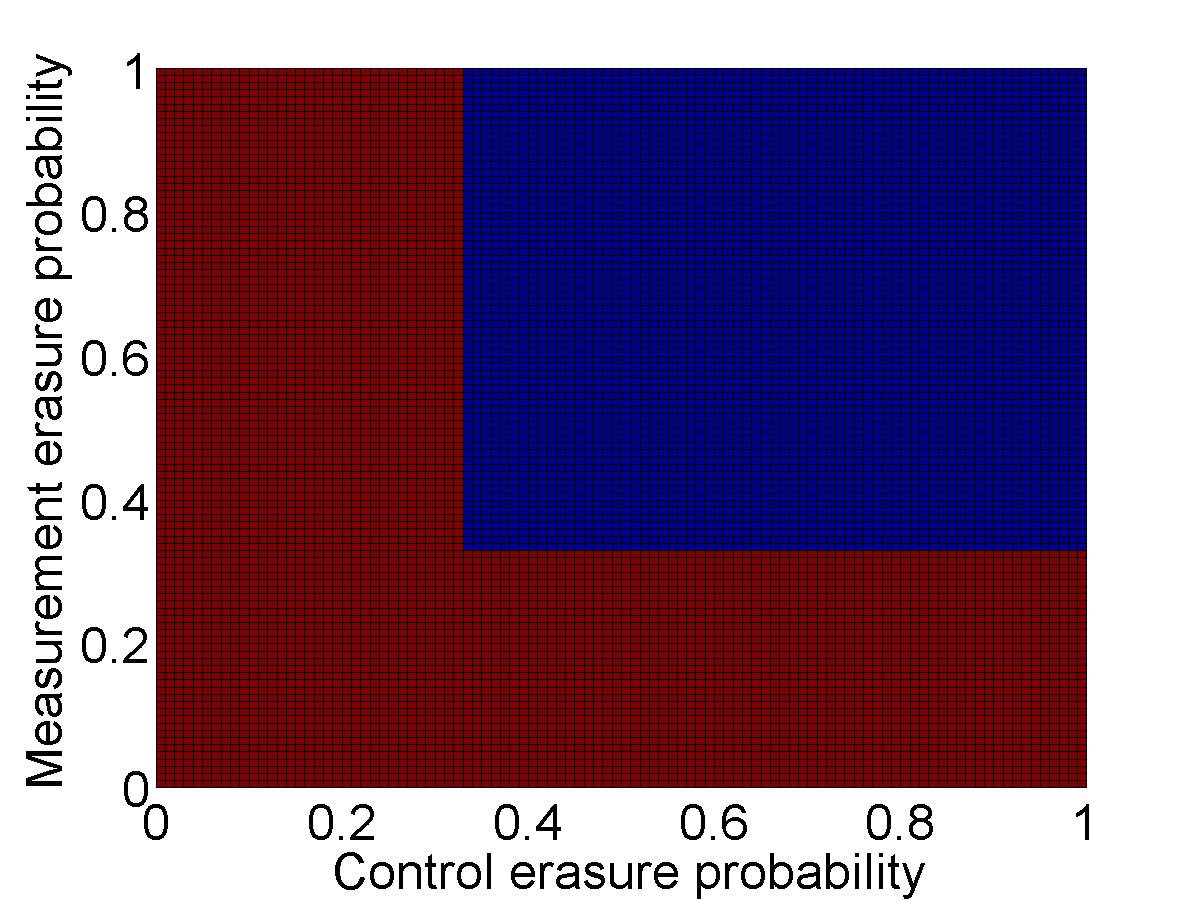}}}
\hspace{-0.05in}
\subfigure[]{\scalebox{0.15}{\includegraphics{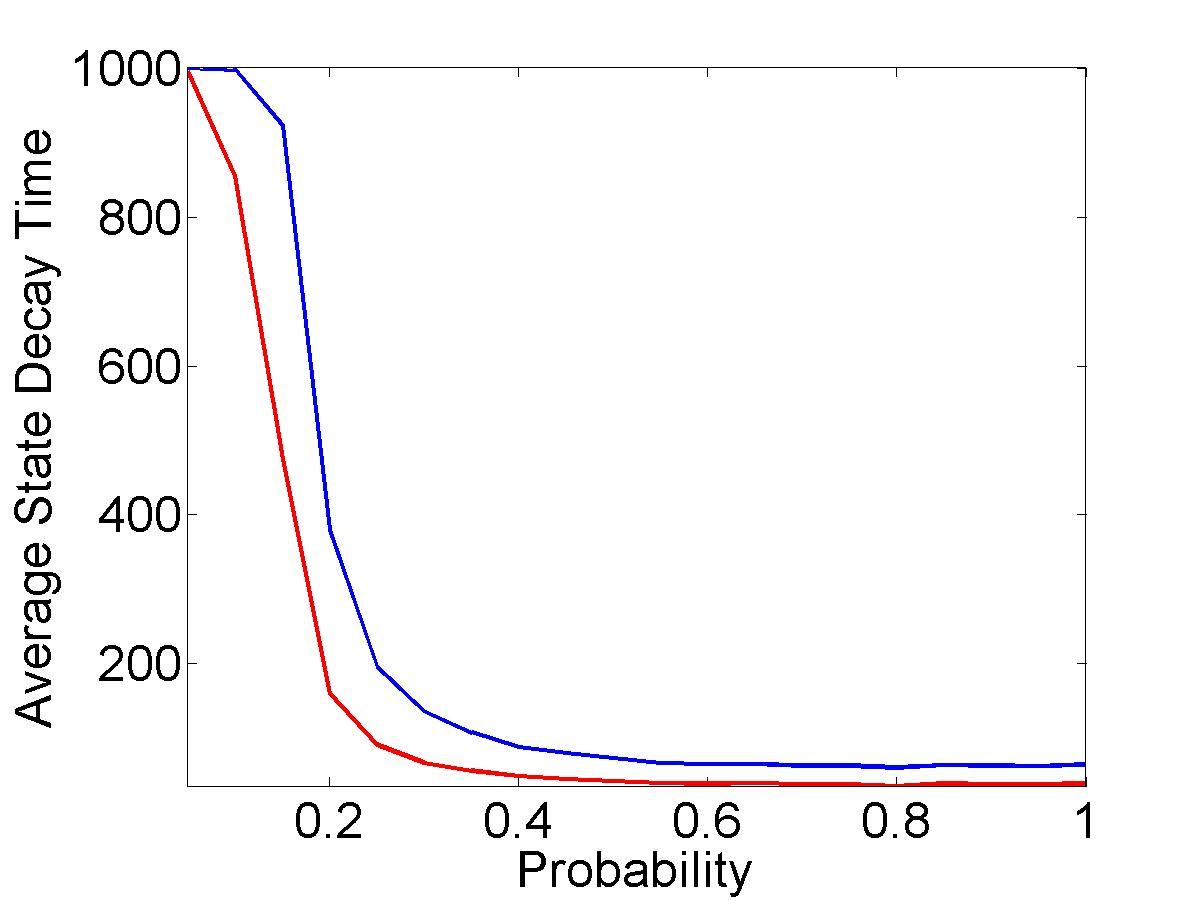}}} }
\caption{(a) Region of acceptable erasure in control and observation - given in blue, (b) Average decay time of controlled state (blue) and observer error (red) vs non-erasure probability}
\label{fig:sim1}
\end{center}
\end{figure}

\begin{figure}[htb!]
\begin{center}
\mbox{
\hspace{-0.05in}
\subfigure[]{\scalebox{0.3}{\includegraphics{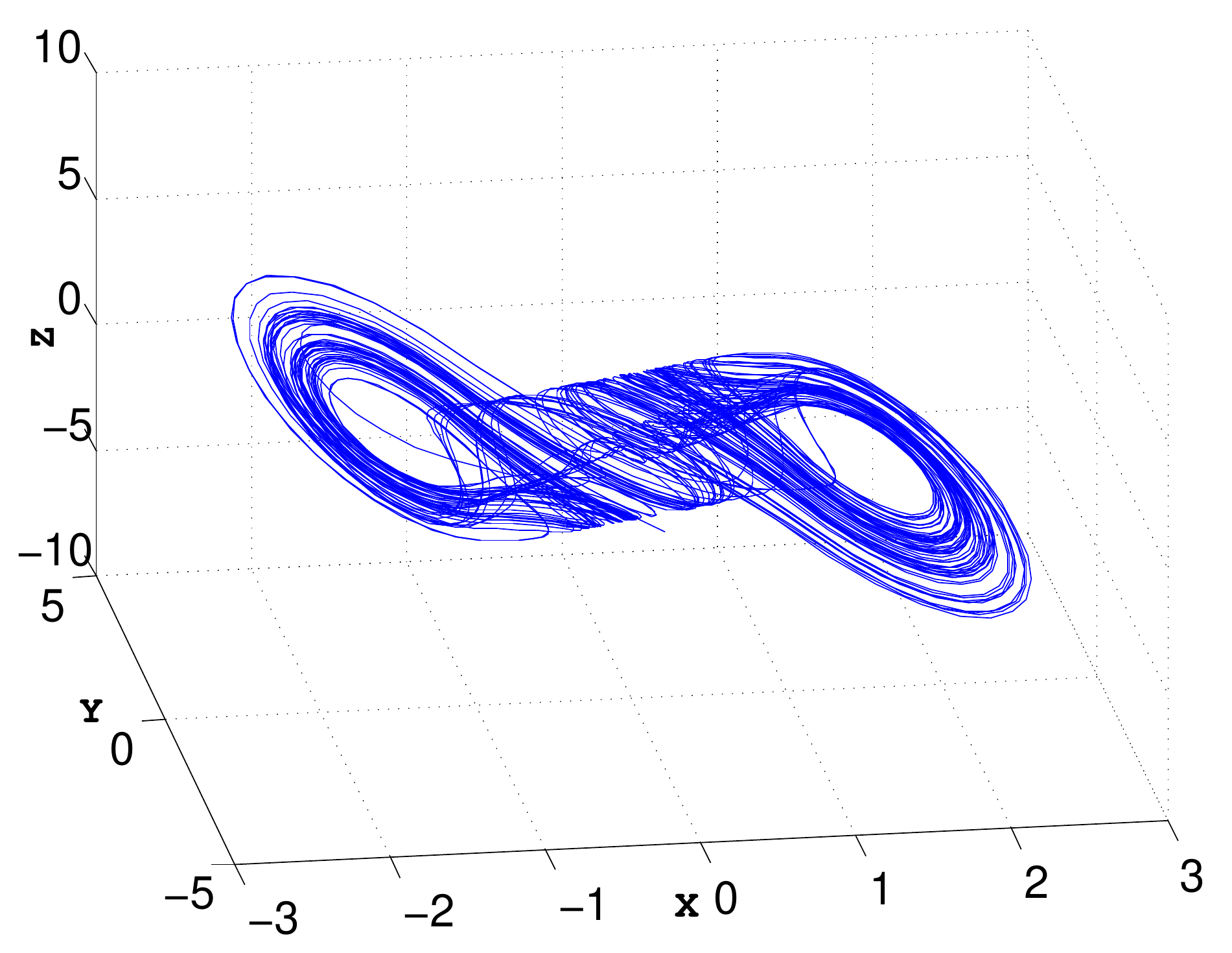}}}
\hspace{-0.05in}
\subfigure[]{\scalebox{0.3}{\includegraphics{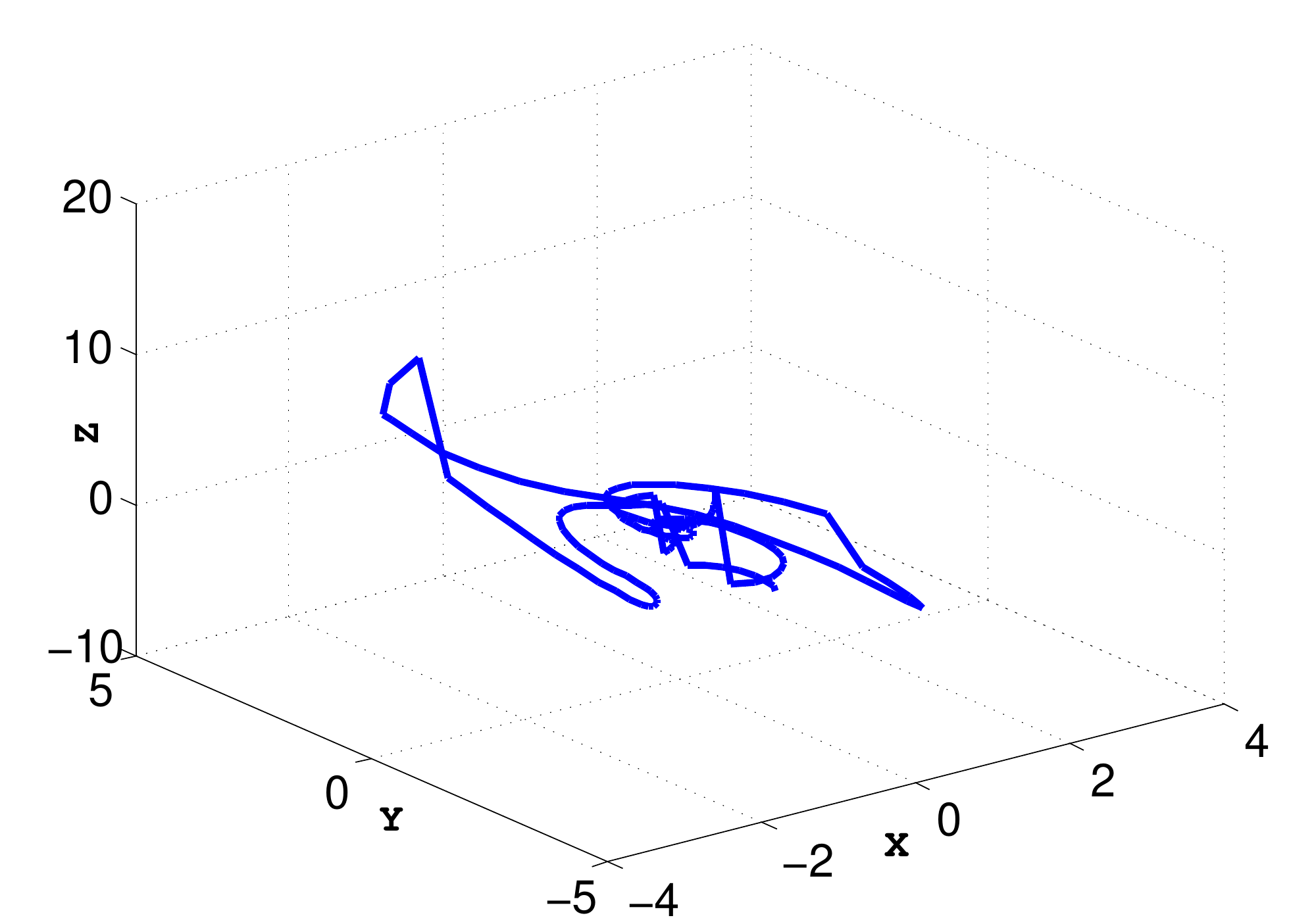}}} }
\caption{ Phase space dynamics for non-erasure probability of (a) $p=0.15$, (b) $p=0.35$}
\label{fig_phase_space}
\end{center}
\end{figure}
The system state and observer error decay to zero for almost all initial conditions and almost all sequences of the uncertainties $\gamma$ and $\xi$. In Fig. (\ref{fig:sim1}a) we plot in blue the region of control and measurement non-erasure probabilities that guarantee mean square stable controlled system and observer dynamics. To observe the effect of high erasure probability on the system we plot the average time required by the system to converge to zero over $100$ realizations of the uncertainty sequence in Fig. (\ref{fig:sim1}b). Shown in red (blue) is average time to decay for the observer error (controlled state) plotted against the non-erasure probability. We observe that the sharp drop in time of decay occurs below the critical non-erasure probability indicating that the system spends significant time away from the origin for probabilities less than the critical probability that guarantees mean square stability. Thus during the time the system is away from the origin, roughly speaking the trajectories move along the chaotic attractor of the uncontrolled system. We believe that this behavior is sensitive to the addition of small amount of additive noise to the system. As any small amount of additive noise will prevent the system from converging to the origin. In Fig. (\ref{fig_phase_space}), we show the phase space dynamics of the system with small amount of additive Gaussian noise with zero mean and $0.1$ variance for two different values of non-erasure probabilities and averaged over $20$ different realizations. Comparing Figs. (\ref{fig_phase_space}a) and (\ref{fig_phase_space}b), we see that the phase space dynamics  clearly reveals the attractor of the system for $p=0.15$ whereas the dynamics for $p=0.35$ is predominantly settled around the origin.

\section{Conclusion}\label{section_conclusion}
We studied the problem of observer-based controller design for nonlinear systems in Lure form over uncertain channels. We derived sufficient condition for the mean square stability of the closed loop system. The results provide for the synthesis method for the design of controller and observer robust to channel uncertainty. The main results of this paper are made possible using the stochastic variant of the Positive Real Lemma and the separation principle for stochastic nonlinear systems. The main result of the paper on mean square stability provide insight as to how the passivity property of the system can be traded off to account for uncertainty in the feedback loop.

\bibliographystyle{IEEETran}
\bibliography{aut_ref}           
\end{document}